\documentclass[12pt]{article}

\usepackage{graphicx}
\usepackage{amssymb}
\usepackage{amsmath}
\usepackage{amsthm}

\newtheorem{theorem}{Theorem}[section]
\newtheorem{lemma}[theorem]{Lemma}
\newtheorem{proposition}[theorem]{Proposition}
\newtheorem{corollary}[theorem]{Corollary}

\theoremstyle{definition}
\newtheorem{definition}[theorem]{Definition}

\theoremstyle{remark}
\newtheorem{remark}[theorem]{Remark}

\numberwithin{equation}{section}

\begin{document}

\title{Generators for the representation rings of certain wreath products}

\author{Nate Harman}



\maketitle

\begin{abstract}
Working in the setting of Deligne categories, we generalize a result of Marin that hooks generate the representation ring of symmetric groups to wreath products of symmetric groups with a fixed finite group or Hopf algebra.  In particular, when we take the finite group to be cyclic order 2 we recover a conjecture of Marin about Coxeter groups in type B.

\end{abstract}

\section{Introduction}

In \cite{B}, Deligne defined the categories $\underline{Rep}(S_t)$ for $t$ an arbitrary complex number. In the context of the Church-Farb framework of representation stability \cite{J}  we may think of these Deligne categories at generic values of $t$ as models for stable categories of representations of the symmetric group.  In particular they satisfy the following ``stable" properties:

\begin{itemize}

\item For generic $t$, $\underline{Rep}(S_t)$ is semisimple with irreducible objects $\tilde{V}(\lambda)$ indexed by partitions.  These interpolate the irreducible representations $V(\lambda(n))$ of $S_n$ with $n>>0$, where we add a sufficiently long first row to $\lambda$ to make it the right size.  In particular these representations are known to have polynomial growth of dimension $Dim(V(\lambda(n))) = p_\lambda(n)$, and in the Deligne category we have $Dim(\tilde{V}(\lambda) )= p_\lambda(t)$.

\item If $k \in \mathbb{Z}_+$ we have induction functors $ Rep(S_k) \boxtimes \underline{Rep}(S_t) \rightarrow  \underline{Rep}(S_{t+k})$, where $Rep(S_k)$ denotes the usual category of complex representations of $S_k$.  The multiplicity $\tilde{c}_{\lambda,\mu}^\nu$ of $\tilde{V}(\nu)$ in ${\rm Ind}(V(\lambda)\boxtimes \tilde{V}(\mu))$ is equal to the stable limit of Littlewood-Richardson coefficients $c_{\lambda,\mu(n)}^{\nu(n+k)}$. Similar statements hold for restriction, with an appropriate version of Frobenius reciprocity.

\item The structure constants for the tensor product are the so called reduced Kronecker coefficients  which are the stable limits of the Kronecker coefficients. 

\end{itemize}

In \cite{A} Marin proves that hooks, i.e. partitions of the form $(n-k,1^k)$ generate the representation ring of $\text{Rep}(S_n)$. While he doesn't use the language, his proof mostly takes place in the stable setting and his argument shows that hooks freely generate the stable representation ring. This then implies they must generate in the classical setting (although not freely).  So we may think of this result as an application of stable representation theory to classical representation theory.

In the Deligne category setting Marin's result appears in Deligne's  original paper \cite{B}, saying that the Grothendieck ring of the Deligne category is freely generated by objects corresponding to hooks.  The result for the classical case follows by projecting from the Deligne category onto $\text{Rep}(S_n)$, and looking at the induced map of Grothendieck rings. 

The proof is done by defining a filtration on the Deligne category such that the associated graded Grothendieck ring is isomorphic in a natural way to the ring $(\bigoplus_n K_0(\text{Rep}(S_n)), \cdot)$  with multiplication coming from inducing representations from $S_n \times S_m$ to $S_{n+m}$.  This ring is well known to be isomorphic to the ring of symmetric functions, and the elementary symmetric functions  correspond to hooks.

Deligne categories for wreath products with a finite group or Hopf algebra were defined by Knop \cite{E}. In \cite{D} Mori defined wreath product Deligne categories associated to an arbitrary $k$-linear category $\mathcal{C}$, which is a tensor category whenever $\mathcal{C}$ is. As for the symmetric group these may be thought of as stable versions of the more classical wreath product categories in ways analogous to those listed above.  See \cite{F} for a more detailed overview of representation theory in complex rank, including discussions of the constructions mentioned here.

Motivated by a conjecture of Marin about generalizing his result to Coxeter groups in type B (which are wreath products), the goal of this paper is to prove similar results about the Grothendieck rings of these Deligne categories.  By projecting from the Deligne categories to classical representation categories, we obtain systems of generators for the representation rings of wreath products with finite groups, answering the conjecture of Marin in the case when the finite group is cyclic of order 2.

\section*{Acknowledgements}
The author is grateful to P. Etingof for suggesting the question, and to V. Ostrik for his many helpful comments on the preliminary version of this paper.  This material is based upon work supported by the National Science Foundation Graduate Research Fellowship. 

\section{The categories $\mathcal{W}_n( \mathcal{C})$ and $S_t(\mathcal{C})$}

First we will review Mori's construction of the wreath product Deligne categories $S_t(\mathcal{C})$, and the aspects of the theory useful for our purposes. This section is largely expository, and all the proofs will be left to the references.  

Fix $k$ an algebraically closed field of characteristic zero, and let $\mathcal{C}$ be a $k$-linear semisimple\footnote{Since we will be mostly working at the level of the Grothendieck group, the semisimple condition can be relaxed to an artinian condition and the main result will still hold but we will assume it for simplicity.} tensor category with a unit object $\textbf{1}$ satisfying $End_{\mathcal{C}} (\textbf{1}) = k$. 

\begin{definition}[\textbf{Mori \cite{D} Definition 3.13}]

Consider the action of the symmetric group $S_n$ on the category $\mathcal{C}^{\boxtimes n}$  by permuting the factors.  We denote the equivariantization of this action by $\mathcal{W}_n( \mathcal{C})$ and refer to it as a wreath product category.
\end{definition}

In the case that $\mathcal{C}$ is the category of representations of a finite group $G$, the category $\mathcal{W}_n(\mathcal{C})$ equivalent to the representation category of the wreath product $S_n(G)$  (also denoted by $G \wr S_n$ or $G^n \rtimes S_n$) of $G$ with a symmetric group. Many of the standard facts about the representation theory of these groups hold in the categorical setting as well, we will highlight some of these properties that are important for our purposes.

\begin{proposition}[\textbf{Mori \cite{D} sections  3.5 and 5.3}]
\

\begin{enumerate}

\item If $I(\mathcal{C})$ is an indexing set for equivalence classes of irreducible objects in $\mathcal{C}$, then the irreducible objects of $\mathcal{W}_n( \mathcal{C})$ are indexed by the set:

$$\mathcal{P}_n^{\mathcal{C}} = \{ \lambda: I(\mathcal{C}) \rightarrow \mathcal{P} \ \text{such that} \ |\lambda| :=\sum_{U\in I(C)}|\lambda(U)| = n \}$$

where $\mathcal{P}$ denotes the set of partitions. We denote the irreducible object corresponding to $\lambda \in \mathcal{P}_n^{\mathcal{C}}$ by $V(\lambda)$.

\item The natural inclusions $S_n \times S_m \hookrightarrow S_{n+m}$ give rise to induction functors $$ {\rm Ind}_{\mathcal{W}_n( \mathcal{C}) \boxtimes \mathcal{W}_m( \mathcal{C})}^{\mathcal{W}_{n+m}(\mathcal{C})} : \mathcal{W}_n( \mathcal{C}) \boxtimes \mathcal{W}_m( \mathcal{C}) \rightarrow \mathcal{W}_{n+m}( \mathcal{C})$$

admitting two sided adjoints ${\rm Res}_{\mathcal{W}_n( \mathcal{C}) \boxtimes \mathcal{W}_m( \mathcal{C})}^{\mathcal{W}_{n+m}(\mathcal{C})}$ referred to as restriction functors. 

\item If we take the Grothendieck ring of $\mathcal{W}_*(\mathcal{C}) := \bigoplus_n \mathcal{W}_n( \mathcal{C})$ with a tensor structure corresponding to induction, then the map

$$[V(\lambda)] \rightarrow \bigotimes_{U \in I(\mathcal{C})} s_{\lambda(U)}$$

is a graded isomorphism with the ring $\bigotimes_{U \in I(\mathcal{C})}\Lambda$ where $\Lambda$ is the ring of symmetric functions, and $s_\lambda$ is a Schur function.

\end{enumerate}
\end{proposition}

We now want to construct categories $S_t(\mathcal{C})$ which interpolate $\mathcal{W}_n( \mathcal{C})$ to non-integer values of $t$ analogously to the construction of the Deligne category $\underline{Rep}(S_t)$ as explained by Comes and Ostrik in \cite{I}. To do this we will define certain well behaved standard objects in $\mathcal{W}_n(\mathcal{C})$ for different values of $n$.

 Let $I = \{i_1,i_2,\dots, i_k\}$ be a finite set and $U_I = (U_i)_{i\in I}$ a collection of objects in $\mathcal{C}$. If $n \ge k$ we can consider $U_{i_1} \boxtimes U_{i_2} \boxtimes ...\boxtimes U_{i_k} \boxtimes \textbf{1}_{\mathcal{C}}^{\boxtimes n-k}$ as an object of  $\mathcal{C}^{\boxtimes k} \boxtimes \mathcal{W}_{n-k}(\mathcal{C})$.  We will let $[U_I]_n \in \mathcal{W}_n( \mathcal{C})$ denote the image of this under the appropriate induction functor, when $n < k$ it will be convenient to define $[U_I]_n$ to be the zero object. We have the following lemma:

\begin{lemma}[\textbf{Mori \cite{D} definition 4.5 and lemma 4.9}]
\

\begin{enumerate}
\item For finite sets $I$ and $J$ and collections of objects $U_I$, $V_J$ as above there exists an explicitly described vector space $H(U_I,V_J)$ such that $Hom_{ \mathcal{W}_n( \mathcal{C})}([U_I]_n,[V_J]_n) \cong  H(U_I,V_J)$ for all $n$ sufficiently large.

\item There exists a map $ \Phi: H(V_I,W_K) \otimes H(U_I,V_J) \rightarrow H(U_I,W_K)\otimes k[x]$ such that for $n$ sufficiently large the composition of $\Phi$ with the evaluation at $n$ map $ev_n: H(U_I,W_K)\otimes k[x] \rightarrow H(U_I,W_K)$  corresponds to the composition map $Hom_{ \mathcal{W}_n( \mathcal{C})}([V_J]_n,[W_K]_n) \otimes Hom_{ \mathcal{W}_n( \mathcal{C})}([U_I]_n,[V_J]_n) \rightarrow Hom_{ \mathcal{W}_n( \mathcal{C})}([U_I]_n,[W_k]_n).$

\end{enumerate}

\end{lemma}

Using this lemma we are able to define an auxiliary category $S_t(\mathcal{C})_0$  for an arbitrary number $t \in k$ as follows:

\noindent \textbf{Objects} are symbols $\langle U_I\rangle_t$ where $U_I$ is a finite collection of objects of $\mathcal{C}$.

\noindent \textbf{Morphisms} from $\langle U_I\rangle_t$ to $\langle V_J\rangle_t$ are given by the space $H(U_I,V_J)$ from the first part of the previous lemma.

\noindent \textbf{Composition} is given by $ev_t \circ \Phi$ where $\Phi$ is the map from the second part of the previous lemma, and $ev_t$ is the evaluation at $t$ map.

\begin{definition}[\textbf{Mori \cite{D} Definitions 4.10 and 4.16}]
The wreath product Deligne category $S_t(\mathcal{C})$ is defined as the pseudo-abelian or Karoubian envelope of $S_t(\mathcal{C})_0$, it comes equipped a natural tensor structure coming from the tensor structure of $\mathcal{C}$.

\end{definition}

Now we want to explicitly describe these Deligne categories $S_t(\mathcal{C})$ and how they interpolate the wreath product categories $\mathcal{W}_n(\mathcal{C})$. First define the set $\mathcal{P}^{\mathcal{C}}$ by:
$$\mathcal{P}^{\mathcal{C}}= \cup_n \mathcal{P}_n^{\mathcal{C}} = \{ \lambda: I(\mathcal{C}) \rightarrow \mathcal{P}| \ \lambda(U) =\emptyset  \ \text{for all but finitely many} \ U\}$$
  and for $\lambda \in\mathcal{P}^{\mathcal{C}}$ and $n$ sufficiently large define $\lambda_n \in \mathcal{P}_n^{\mathcal{C}}$ by:
    
\begin{equation}{\label{Z}}
   \lambda_n(U) = \left\{
     \begin{array}{lr}
       \lambda(U) & : U \ne \textbf{1}_\mathcal{C}\\
       (n-|\lambda|, \lambda(U)) & : U = \textbf{1}_\mathcal{C}
     \end{array}
   \right.
\end{equation}

\noindent That is $\lambda_n$ is obtained from $\lambda$ by adding a long first row to the partition corresponding to the unit object of $\mathcal{C}$, and leaving the rest the same.  If $n$ is not sufficiently large, i.e. if $(n-|\lambda|, \lambda(\textbf{1}))$ is not a valid partition, then it will be convenient to define $\lambda_n = 0$.  We have the following description of $S_t(\mathcal{C})$:

\begin{theorem}[\textbf{Mori \cite{D} Theorems 4.13 and 5.6}]
\

\begin{enumerate}

\item There exists a bijection between $\mathcal{P}^{\mathcal{C}}$ and indecomposable objects of $S_t(\mathcal{C})$.  We denote these corresponding indecomposable objects by $\tilde{V}(\lambda)$.

\item When $t \notin \mathbb{Z}_{\ge0}$ the category $S_t(\mathcal{C})$ is semisimple.

\item For $n \in \mathbb{Z}_{\ge0}$ there exists a full, essentially surjective tensor functor $S_n(\mathcal{C}) \rightarrow \mathcal{W}_n(\mathcal{C})$ which we will refer to as a projection functor.

\item For all $\lambda \in \mathcal{P}^{\mathcal{C}}$ there exists $N$ such that the image of $\tilde{V}(\lambda)$ under the projection from $S_n(\mathcal{C})$ to $\mathcal{W}_n(\mathcal{C})$ is isomorphic to $V(\lambda_n)$ for all $n \ge N$.

\end{enumerate}
\end{theorem}

In order to pass results from $S_t(\mathcal{C})$ at generic values of $t$ to the more classical wreath product categories $\mathcal{W}_n(\mathcal{C})$ we will need a more detailed description of what happens at integer values of $t$.  For our purposes however it will be convenient to do this after we have defined a filtration in the next section.

\begin{remark}
 These sequences of irreducible representations $V(\lambda_n)$ of $S_n(G)$ are exactly those that show up in finitely generated $FI_G$ modules, as defined by Gan and Li in \cite{M} and developed further by Sam and Snowden in \cite{SS}. This is consistent with the philosophy that these Deligne categories can be thought of as models for a stable representation category.
\end{remark}

\subsection{Induction and the $|\lambda|$ filtration} 





An immediate consequence of Mori's construction in terms of the objects $[U_I]$ is the existence of induction functors $\mathcal{W}_n( \mathcal{C}) \boxtimes S_t(\mathcal{C}) \rightarrow S_{t+n}(\mathcal{C})$ which interpolate the induction functors $\mathcal{W}_n( \mathcal{C}) \boxtimes \mathcal{W}_m( \mathcal{C})\rightarrow \mathcal{W}_{n+m}( \mathcal{C})$ where we fix $n$ and let $m$ grow to infinity. This defines a categorical action of the tensor category $\mathcal{W}_*( \mathcal{C})$ on $S_*(\mathcal{C}) := \bigoplus_{t\in k} S_t(\mathcal{C})$.  

 Similarly, we get restriction functors the other direction which are two sided adjoints to the induction functors and interpolate the corresponding restriction functors between the genuine wreath product categories.

Mori's construction ensures that objects of $S_{t}(\mathcal{C})$  occur as summands in objects of the form: 
\begin{equation}{\label{A}}
{\rm Ind}_{\mathcal{W}_k( \mathcal{C}) \boxtimes S_{t-k}(\mathcal{C})}^{S_{t}(\mathcal{C})}(W\boxtimes \textbf{1}_{S_{t-k}(\mathcal{C})})
\end{equation}
for some natural number $k$ and $W \in \mathcal{W}_k( \mathcal{C})$.   In particular we could describe $\tilde{V}(\lambda)$  as the unique indecomposable summand of: $$M(\lambda) := {\rm Ind}_{\mathcal{W}_n( \mathcal{C}) \boxtimes S_{t-n}(\mathcal{C})}^{S_{t}(\mathcal{C})}(V(\lambda) \boxtimes \textbf{1}_{S_{t-n}(\mathcal{C})})$$ not occurring as a summand in an object of the form (\ref{A}) for any $k <n$. This suggests we consider the following filtrations:

\begin{definition}[The $|\lambda|$-filtration]
Define a filtration on $S_{t}(\mathcal{C})$ by putting an object $M$ in the $k$th level of the filtration if $M$ occurs as a summand in an object of the form of (\ref{A}). Analogously define such filtrations on the categories $\mathcal{W}_n(\mathcal{C})$.

\end{definition}

Immediately we see that $\tilde{V}(\lambda)$ is minimally contained in the $|\lambda|$th level of the filtration on $S_{t}(\mathcal{C})$. Similarly $V(\lambda)$ is minimally contained in the $(n - \mu( \textbf{1})_1)$th level of the filtration on $\mathcal{W}_n(\mathcal{C})$.  


By the nature of its definition, this filtration is well behaved with respect to the induction functors.  If we let $c_{\lambda, \mu}^\gamma$ denote the structure constants for $(\bigoplus_n K_0(\mathcal{W}_n(\mathcal{C})), \cdot)$, which are just products of the usual Littlewood-Richardson coefficients by Proposition 2.2 part 3. Then by looking at the Littlewood-Richardson rule with one partition having a long first row, we get the following relation:

\begin{equation}{\label{B}}
{\rm Ind}_{\mathcal{W}_n( \mathcal{C}) \boxtimes S_{t-n}(\mathcal{C})}^{S_{t}(\mathcal{C})}(V(\lambda) \boxtimes \tilde{V}(\mu))
 = (\bigoplus c_{\lambda, \mu}^\gamma \tilde{V}(\gamma) )\oplus \{ \text{terms lower in the filtration} \}
\end{equation}

Frobenius reciprocity and similar analysis of the Littlewood-Richardson rule gives us the ``lead term" for restrictions of indecomposables:

\begin{equation}{\label{C}}
{\rm Res}_{\mathcal{W}_n( \mathcal{C}) \boxtimes S_{t-n}(\mathcal{C})}^{S_{t}(\mathcal{C})}(\tilde{V}(\mu)) = (\textbf{1} \boxtimes \tilde{V}(\mu)) \oplus \{ \text{terms} \ M \boxtimes \tilde{V}(\nu) \text{ with } |\nu| < |\mu| \}
\end{equation}

In the case of $\mathcal{C} =Vect_k$ this $S_t(\mathcal{C})$ is equivalent to $\underline{Rep}(S_t)$, and our filtration agrees with the filtration defined by Deligne in \cite{B} and by Marin in \cite{A}.  In that case the filtration was also well behaved with respect to the internal tensor structure, and we had that the associated graded Grothendieck ring was isomorphic to the ring of symmetric functions.  In our setting we have the following generalization:

\begin{lemma}
The filtration defined above gives the Grothendieck ring $K_0(S_t(\mathcal{C}))$ the structure of a filtered ring.  Moreover, the associated graded ring is then isomorphic to $(\bigoplus_n K_0(\mathcal{W}_n(\mathcal{C})), \cdot)$, where the isomorphism sends the image of $[\tilde{V}(\lambda)]$ to $[V(\lambda)]$.
\end{lemma}

\begin{proof}
Let $M(\lambda) = {\rm Ind}_{\mathcal{W}_n( \mathcal{C}) \boxtimes S_{t-n}(\mathcal{C})}^{S_{t}(\mathcal{C})}(V(\lambda) \boxtimes \textbf{1}_{S_{t-n}(\mathcal{C})})$.  At the level of the Grothendieck ring we have that $[M(\lambda)] = [\tilde{V}(\lambda)] + \{ \text{terms lower in the filtration}\}$.  So inductively we can conclude that $[\tilde{V}(\mu) \otimes \tilde{V}(\lambda)]$ and  $[\tilde{V}(\mu) \otimes M(\lambda)]$ have the same highest order terms with respect to this filtration. To simplify the notation we will now begin omitting the subscripts and superscripts from the induction and restriction functors, they all will go between $\mathcal{W}_n( \mathcal{C}) \boxtimes S_{t-n}(\mathcal{C})$ and $S_{t}(\mathcal{C})$;

$$\tilde{V}(\mu) \otimes M(\lambda) = \tilde{V}(\mu) \otimes {\rm Ind}(V(\lambda) \boxtimes  \textbf{1})$$

$$ = {\rm Ind} ({\rm Res}(\tilde{V}(\mu)) \otimes (V(\lambda) \boxtimes  \textbf{1}))$$

By (\ref{C}) this becomes:

$$  \tilde{V}(\mu) \otimes M(\lambda) = {\rm Ind} (((\textbf{1} \boxtimes \tilde{V}(\mu) \oplus \{ \text{terms} \ M \boxtimes \tilde{V}(\nu) \text{ with} \  |\nu| < |\mu| \}) \otimes (V(\lambda) \boxtimes  \textbf{1}) )$$

$$  \ \ \ \ \ \   \  = {\rm Ind}( (V(\lambda) \boxtimes \tilde{V}(\mu)) \oplus \{ \text{terms} \ M \boxtimes \tilde{V}(\nu) \text{ with} \  |\nu| < |\mu| \} )$$

Which by (\ref{B}) becomes:

$$ \tilde{V}(\mu) \otimes M(\lambda) =  \bigoplus c_{\lambda, \mu}^\gamma \tilde{V}(\gamma) \oplus \{ \text{summands lower in the filtration} \}$$

Since the coefficients $c_{\lambda, \mu}^\gamma$ were the structure constants for the induction product, we see that indeed the associated graded Grothendieck ring of $S_{t}(\mathcal{C})$ is isomorphic to the induction ring $(\bigoplus_n K_0(\mathcal{W}_n(\mathcal{C})), \cdot)$.

\end{proof}

This combined with proposition 2.2 part 3 immediately gives us the following useful corollary:

\begin{corollary} \label{isom}
The associated graded Grothendieck ring of $S_{t}(\mathcal{C})$ at generic $t$ with respect to the $|\lambda|$-filtration is isomorphic to $\bigotimes_{U \in I(\mathcal{C})}\Lambda$.  The map sends irreducible objects to products of Schur functions.

\end{corollary}

\begin{section}{Deligne categories at integer $t$}

Our goal is to obtain a collection of generators for the representation rings of wreath products $S_n(G)$, or more generally the Grothendieck rings of the categorical wreath products $\mathcal{W}_n(\mathcal{C})$. To do this we want to take a nice system of generators for these relatively well behaved Deligne categories and pass them down to these more classical categories.  

In general some care needs to be taken in the Deligne category setting when taking $t$ to be a positive integer.  What happens in this case is handled in depth in the case of symmetric groups by Comes and Ostrik in \cite{I}, and their results were extended to the wreath product setting by Mori \cite{D}. We will outline the parts of the theory that are relevant for our purposes.

The Deligne category $S_{t=n}(\mathcal{C})$ fails to be semisimple or even abelian, but it has a tensor structure and we can still consider the split Grothendieck ring spanned by indecomposable objects. We want to think of these as being a link between the Grothendieck rings of $\mathcal{W}_n(\mathcal{C})$, and the Grothendieck rings of $S_{t}(\mathcal{C})$ at generic $t$.  We can make this precise via projection and lifting.

\begin{theorem}[\textbf{Description of projection, Mori \cite{D} theorem 5.6}]

 The projection functor of theorem 2.5 induces a surjective homomorphism from this split Grothendieck ring to the Grothendieck ring of $\mathcal{W}_n(\mathcal{C})$. Explicitly, it sends  $\tilde{V}_{t=n}(\lambda)$ to either $V(\lambda_n)$  if $n - |\lambda| \ge \lambda(\textbf{1}_\mathcal{C})_1$ or the zero object otherwise. 

\end{theorem}

In particular we have the following immediate corollary:

\begin{corollary}

A collection of generators of the split Grothendieck ring of  $S_{t=n}(\mathcal{C})$ consisting of elements of the form $[\tilde{V}_{t=n}(\lambda)]$ projects to a collection of generators of the Grothendieck ring of $\mathcal{W}_n(\mathcal{C})$ consisting of elements of the form $[V(\lambda_n)]$.

\end{corollary}

 

Next we want to relate the split Grothendieck rings at positive integer $t$ to the Grothendieck ring at generic $t$. We need to be a bit careful because the indecomposable objects $\tilde{V}_{t=n}(\lambda)$ of $S_{t=n}(\mathcal{C})$ are not always flat deformations of the irreducible objects $\tilde{V}(\lambda)$ in the Deligne categories for generic $t$. In particular we do not expect the map $[\tilde{V}_{t=n}(\lambda)] \mapsto [\tilde{V}(\lambda)]$ to be a homomorphism of between these rings. We can relate these by the process of lifting, which has the following (simplified) description:

\begin{theorem}[\textbf{Description of lifting, after \cite{I} lemma 5.20}] \

There exists a map Lift$_t$ from objects of $S_{t=n}(\mathcal{C})$ to objects of $S_T(\mathcal{C})$ where $T$ is a formal variable, satisfying:

\begin{enumerate}
\item Lift$_t$ descends to an isomorphism from the split Grothendieck ring $S_{t=n}(\mathcal{C})$ to the Grothendieck ring of $S_T(\mathcal{C})$.

\item This isomorphism sends $[\tilde{V}_{t=n}(\lambda)]$ to either $[\tilde{V}(\lambda)]$ or $[\tilde{V}(\lambda)]+[\tilde{V}(\lambda ')]$ for an explicitly described $\lambda'$ depending on the combinatorics of $n$ and $\lambda$ satisfying $|\lambda'| < |\lambda|$.

\end{enumerate}

\end{theorem}

\begin{proof}
The existence of a lift follows from standard facts about lifting of idempotents, and the proof is identical to that for $\underline{Rep}(S_t)$.  The combinatorics follows from the case of $\underline{Rep}(S_t)$ via an explicit identification of the blocks of $\mathcal{S}_t(\mathcal{C})$ with blocks of  $\underline{Rep}(S_{t'})$ given by Mori (\cite{D}, Prop 5.4).

\end{proof}

In particular, the additional ``error" term picked up when lifting is lower in the filtration.  So while the map  $[\tilde{V}_{t=n}(\lambda)] \mapsto [\tilde{V}(\lambda)]$ is not an isomorphism (or even a homomorphism) between the split Grothendieck ring at integer $t$ and the Grothendieck ring at generic $t$, it is an isomorphism at the level of associated graded rings.  We now have the following key lemma:

\begin{lemma} \label{main}
If $\{[\tilde{V}(\lambda)] \ | \ \lambda \in J \}$ is a collection of generators for the associated graded Grothendieck ring of $S_t(\mathcal{C})$ at generic values of $t$ for some indexing set $J \subset \mathcal{P}^{\mathcal{C}}$, then $\{ [V(\lambda_n)] \ | \ \lambda \in J \}$ is a collection of generators for the Grothendieck ring of $\mathcal{W}_n(\mathcal{C})$.

\end{lemma} 

\begin{proof}

The previous remark allows us to transfer this to a system of generators for the associated graded Grothendieck ring at integer $t$.  A standard upper triangularity argument tells us that any lift of this system of generators will also generate the split Grothendieck ring. Corollary 3.2 allows us to transfer this down to  $\mathcal{W}_n(\mathcal{C})$.

\end{proof}

\begin{remark} Understanding these non-semisimple Deligne categories and their split Grothendieck rings in better detail is of independent interest as they seem to lie somewhere between classical representation theory and stable representation theory. Recently Inna Entova-Aizenbud \cite{H} used the non-semisimple Deligne categories for symmetric groups to find new identities involving the reduced and non-reduced Kronecker coefficients.

\end{remark}

\end{section}

\section{Explicit systems of generators}

Our goal is to write down explicit systems of generators for the Grothendieck rings of the wreath product categories $\mathcal{W}_n(\mathcal{C})$ consisting of irreducible objects.  By lemma \ref{main} it suffices to find an explicit system of generators for the associated graded Grothendieck ring of $S_t(\mathcal{C})$ at generic $t$.  By corollary \ref{isom} this is isomorphic to the tensor product of one copy of the ring $\Lambda$ of symmetric functions for each irreducible object of $\mathcal{C}$.

Once translated into the language of symmetric functions, for each irreducible object of $\mathcal{C}$ we need to choose a collection of Schur functions which generate the ring of symmetric functions.  Two well known such collections are the elementary and complete homogeneous symmetric functions, which are Schur functions corresponding to $\lambda = (1^k)$ or $\lambda = (k)$ respectively.  All that is left is to describe which objects these correspond to.



First let's handle the case when the irreducible object of $\mathcal{C}$  we are looking at is the unit object.  The elementary and complete homogeneous symmetric functions just correspond to the usual hook and length two partition representations of the natural copy of $\underline{Rep}(S_t)$ inside $S_t(\mathcal{C})$.   If $\mathcal{C}$ is the representation category of a finite group $G$ and $t = n$ is a natural number then under the projection to the category of representations of $S_n(G)$ these are just the corresponding representations of $S_n$ with trivial action of $G$.




If $V \in I(\mathcal{C})$ is a nontrivial irreducible then let $V^{k,\varepsilon}  := V^{\boxtimes k}\otimes \varepsilon \in \mathcal{W}_k(\mathcal{C})$ where $\varepsilon$ is either the sign or trivial representation for the copy of $Rep(S_k)$ in $\mathcal{W}_k(\mathcal{C})$.   Under our correspondence the elementary (resp. complete homogenous) symmetric functions correspond to the objects ${\rm Ind}_{\mathcal{W}_k(\mathcal{C})\boxtimes S_{t-k}(\mathcal{C})}^{S_{t}(\mathcal{C})}(V^{k,\varepsilon} \boxtimes  \textbf{1})$ for $\varepsilon$ the sign representation (resp. trivial).




Putting it all together in the case where $\mathcal{C}$  is the representation category of a finite group $G$ with $m$ equivalence classes of irreducible representations, we get $2^m$ easy to describe systems of generators for the representation ring of the wreath product $S_n(G)$. The previous remarks imply the following result:

\begin{theorem}{\textbf{(Systems of generators for representation rings)}}

If $G$ is a finite group, then the representation ring of the wreath product $S_n(G)$ is generated by either the hook or length two partition representations of $S_n$ along with induced representations ${\rm Ind}_{S_k(G)\times S_{n-k}(G)}^{S_n(G)}( (V^{\otimes k} \otimes \varepsilon_V) \boxtimes \textbf{1}_{S_{n-k}(G)})$ for each irreducible representation $V$ of $G$ and $k\le n$, where $\varepsilon_V$ is a consistent (not depending on $k$) choice of either the trivial or sign character of the symmetric group $S_k$.

\end{theorem}

\begin{subsection}{The abelian group case and Marin's conjecture}

Of particular interest is the case when $G$ is an abelian group, this includes Coxeter groups in type B as well as the complex reflection groups $G(m,1,n)$.  Here we can get an alternate description of the objects corresponding to the elementary symmetric functions.

The sign-twisted representations ${\rm Ind}_{S_k(G)\times S_{n-k}(G)}^{S_n(G)}( (V^{\otimes k} \otimes \varepsilon)\boxtimes \textbf{1}_{S_{n-k}(G)})$ described in the previous section appear as summands in exterior powers of the relatively easy to describe representations ${\rm Ind}_{G \times S_{n-1}(G)}^{S_n(G)}( V \boxtimes \textbf{1}_{S_{n-1}(G)})$, along with other more complicated terms coming from the exterior powers of $V$ itself.  However if we start with a nontrivial character $\chi$ of $G$ these other terms all vanish and these exterior powers coincide with our sign-twisted representations.

 These exterior powers are perhaps closer to a direct generalization of hooks for these wreath products, and can be taken to be in our generating sets by choosing the elementary symmetric functions over the complete homogeneous symmetric functions for each character of $G$.  In particular if $G$ is abelian then all of its irreducible representations are characters and Theorem 4.1 becomes:
 
\begin{theorem}{\textbf{(Hook-like generating sets)}}

If $G$ is a finite abelian group, then the representation ring of the wreath product $S_n(G)$ is generated by the reflection representation of $S_n$, the $n$-dimensional representations ${\rm Ind}_{G \times S_{n-1}(G)}^{S_n(G)}( \chi \boxtimes \textbf{1}_{S_{n-1}(G)})$ for nontrivial characters $\chi$ of $G$, and exterior powers thereof.

\end{theorem}

\begin{remark}
While this paper was being written the author was informed that this result for wreath products with Abelian groups has been recently proven by Schlank and Stapleton using different methods.
\end{remark}

Now let's specialize to the case when $G$ is cyclic of order 2.  The wreath products $S_n(G)$ are Coxeter groups $W$ of type $B_n$.  If $\chi$ is the nontrivial representation of $G$, then the reflection representation of $W$ is given by $V := {\rm Ind}_{G\times S_{n-1}(G)}^{S_n(G)} (\chi \boxtimes \textbf{1}_{S_{n-1}(G)})$.  Next we let $U$ be the reflection representation of $S_n$, upgraded to a $S_n(G)$ representation by letting $G$ act trivially. Translated into this language our theorem becomes:

\begin{theorem}{\textbf{(Marin's conjecture 6.2 for type $B_n$)}}

For $W$ a Coxeter group of type $B_n$ the representation ring $R(W)$ is generated by $\Lambda^k U, \Lambda^k V,  k \ge 0$.

\end{theorem}

\end{subsection}

\bibliographystyle{amsplain}

\end{document}